\documentclass[12pt,reqno]{article}

\usepackage[usenames]{color}
\usepackage{amssymb}
\usepackage{graphicx}
\usepackage{amscd}

\usepackage{amsthm}
\newtheorem{theorem}{Theorem}

\newtheorem{proposition}[theorem]{Proposition}
\newtheorem{corollary}[theorem]{Corollary}

\theoremstyle{definition}

\newtheorem{example}[theorem]{Example}

\usepackage[colorlinks=true,
linkcolor=webgreen, filecolor=webbrown,
citecolor=webgreen]{hyperref}

\definecolor{webgreen}{rgb}{0,.5,0}
\definecolor{webbrown}{rgb}{.6,0,0}

\usepackage{color}

\usepackage{float}

\usepackage{graphics,amsmath,amssymb}
\usepackage{amsfonts}
\usepackage{latexsym}
\usepackage{epsf}

\newcommand{\seqnum}[1]{\href{http://www.research.att.com/cgi-bin/access.cgi/as/~njas/sequences/eisA.cgi?Anum=#1}{\underline{#1}}}

\begin{document}

\begin{center}
\vskip 1cm{\LARGE\bf On the $r$-shifted central triangles of a Riordan array} \vskip 1cm \large
Paul Barry\\
School of Science\\
Waterford Institute of Technology\\
Ireland\\
\href{mailto:pbarry@wit.ie}{\tt pbarry@wit.ie}
\end{center}
\vskip .2 in

\begin{abstract} Let $A$ be a proper Riordan array with general element $a_{n,k}$. We study the one parameter family of matrices whose general elements are given by $a_{2n+r, n+k+r}$. We show that each such matrix can be factored into a product of a Riordan array and the original Riordan array $A$, thus exhibiting each element of the family as a Riordan array. We find transition relations between the elements of the family, and examples are given. Lagrange inversion is used as a main tool in the proof of these results. \end{abstract}

\section{Preliminaries}

The authors Zheng and Yang \cite{Zheng} have considered the family of $r$-shifted central coefficients $a_{2n+r, n+r}$ of a Riordan array $A$. In this note, we consider the family of matrices $c(A;r)$ with general term  $a_{2n+r, n+k+r}$ that have these $r$-shifted central coefficients as their leftmost column. We shall call these matrices the $r$-shifted central triangles of the Riordan array $A$. By construction, these are lower triangular matrices with $1$'s on the principal diagonal. In the sequel we shall  characterize this family as a one parameter family of Riordan arrays.

We let $A$ denote a Riordan array \cite{SGWW}, defined by
$$A=(g(x), xf(x)),$$ where $g(x)$ and $f(x)$ are power series with integer coefficients, which satisfy
$$g(0)=1,\quad f(0)=1.$$
The general term $a_{n,k}$ of the Riordan array $A$ is then defined by
$$a_{n,k}=[x^n] g(x)(xf(x))^k,$$ where the operator $[x^n]$  extracts the coefficient of $x^n$ \cite{Merlini_MC}.

For a power series $h(x)$ with $h(0)=0$, we call the solution $u$ to the equation
$$h(u)=x,$$ which satisfies $u(0)=0$ the reversion of the power series $h(x)$. We denote this
by $$\bar{h}(x), \quad \textrm{or} \quad \textrm{Rev}(h)(x).$$

We recall that the set of Riordan arrays is a group, where the product of two Riordan arrays is defined by
$$ (g(x), xf(x)) \cdot (u(x), xv(x)) = (g(x)u(xf(x)), xf(x) v(xf(x))),$$ and the inverse of a Riordan array is given by $$ (g(x), xf(x))^{-1}=\left(\frac{1}{g\left(\overline{xf(x)}\right)}, \overline{x f(x)}\right).$$

\noindent We shall need the following version of Lagrange inversion in the sequel \cite{Merlini_When, Stanley_2}.
\begin{theorem} (Lagrange-B\"urmann Inversion). Suppose that a formal power series $w = w(t)$
is implicitly defined by the relation $w = t \phi(w)$, where $\phi(t)$ is a formal power series such that
$\phi(0) \ne 0$. Then, for any formal power series $F(t)$,
$$[t^n]F(w(t))=\frac{1}{n}[t^{n-1}] F'(t)(\phi(t))^n.$$
\end{theorem}
\noindent A consequence of this is that if $v(x)=\sum_{n\ge 0}v_n x^n$ is a power series with $v_0=0$, $v_1 \ne 0$, we have
$$[x^n] F(Rev(v))=\frac{1}{n} [x^{n-1}] F'(x) \left(\frac{x}{v}\right)^n.$$

An important feature of Riordan arrays is that they have a number of sequence characterizations \cite{Cheon, He}. The simplest of
these
is as follows.
\begin{proposition} \label{Char} \cite[Theorem 2.1, Theorem 2.2]{He} Let $D=[d_{n,k}]$ be an infinite triangular matrix. Then $D$ is a Riordan array if and only if there
exist two sequences $A=[a_0,a_1,a_2,\ldots]$ and $Z=[z_0,z_1,z_2,\ldots]$ with $a_0 \neq 0$, $z_0 \neq 0$ such that
\begin{itemize}
\item $d_{n+1,k+1}=\sum_{j=0}^{\infty} a_j d_{n,k+j}, \quad (k,n=0,1,\ldots)$
\item $d_{n+1,0}=\sum_{j=0}^{\infty} z_j d_{n,j}, \quad (n=0,1,\ldots)$.
\end{itemize}
\end{proposition}
The coefficients $a_0,a_1,a_2,\ldots$ and $z_0,z_1,z_2,\ldots$ are called the $A$-sequence and the $Z$-sequence of the Riordan array
$D=(g(x),f(x))$, respectively.
Letting $A(x)$ be the generating function of the $A$-sequence and $Z(x)$ be the generating function of the $Z$-sequence, we have
\begin{equation}\label{AZ_eq} A(x)=\frac{x}{\bar{f}(x)}, \quad Z(x)=\frac{1}{\bar{f}(x)}\left(1-\frac{1}{g(\bar{f}(x))}\right).\end{equation}
Information about this sequence characterization is captured in the production matrix of a Riordan array.
In general, for an invertible matrix $A$, its production matrix is defined to be the matrix
$$P=A^{-1}\cdot \bar{A},$$ where $\bar{A}$ is the matrix $A$ with its top row removed. The matrix $P$ will be a (lower) Hessenberg matrix. In the other direction, if we are given a lower Hessenberg matrix $P$, we can use it to generate a lower triangular matrix as follows. Letting
$\mathbf{r}_0$
be the row vector
$$\mathbf{r}_0=(1,0,0,0,\ldots),$$ we define $\mathbf{r}_i=\mathbf{r}_{i-1}P$, $i \ge 1$.
Stacking these rows leads to another infinite matrix which we
denote by
$A_P$. Then $P$ is said to be the \emph{production matrix} for
$A_P$. We have the following result.
 \begin{proposition} \label{RProdMat}
\cite[Proposition 3.1]{ProdMat}\label{AZ} Let $P$ be
an infinite production matrix and let $A_P$ be the matrix
induced by $P$. Then $A_P$ is an (ordinary) Riordan matrix if
and only if $P$ is
of the form \begin{displaymath} P=\left(\begin{array}{ccccccc}
\xi_0 & \alpha_0 & 0 & 0 & 0 & 0 & \ldots \\\xi_1 & \alpha_1 &
\alpha_0 & 0 &
0 & 0 & \ldots \\ \xi_2 & \alpha_2 & \alpha_1 & \alpha_0 & 0 &
0 & \ldots \\ \xi_3 & \alpha_3 & \alpha_2 & \alpha_1 &
\alpha_0
& 0 & \ldots
\\ \xi_4 & \alpha_4 & \alpha_3 & \alpha_2 & \alpha_1 &
\alpha_0
& \ldots \\\xi_5 & \alpha_5  & \alpha_4 & \alpha_3 & \alpha_2
&
\alpha_1
&\ldots\\ \vdots & \vdots & \vdots & \vdots & \vdots & \vdots
&
\ddots\end{array}\right)\end{displaymath} where $\xi_0 \neq 0$, $\alpha_0 \neq 0$. Moreover, columns $0$
and $1$ of
the matrix $P$ are the $Z$- and $A$-sequences,
respectively, of the Riordan array $A_P$. \end{proposition}

Many interesting examples of sequences and Riordan arrays can be found in Neil Sloane's On-Line
Encyclopedia of Integer Sequences (OEIS), \cite{SL1, SL2}. Sequences are frequently referred to by their
OEIS number. For instance, the binomial matrix $\mathbf{B}=\left(\frac{1}{1-x}, \frac{x}{1-x}\right)$ (``Pascal's triangle'') is \seqnum{A007318}.

\section{Initial results}
The central coefficients of the Riordan array $A$ are the elements $a_{2n,n}$. We shall call the lower-triangular matrix $c(A;1)$ with general element $a_{2n,n+k}$ the central triangle of the Riordan array $A$.
\begin{example} We consider the Riordan array $\left(\frac{1}{1-x}, \frac{x(1+x)}{1-x}\right)$. This triangle begins
\begin{displaymath}A=\left(\begin{array}{ccccccc} \mathbf{1} & 0 & 0 & 0
&0 & 0 & \cdots \\1 & 1 & 0 & 0 & 0 & 0 & \cdots \\ 1 & \mathbf{3} & \mathbf{1} &
0 & 0 & 0 &
\cdots \\ 1 & 5 & 5 & 1 & 0 & 0 & \cdots \\ 1 & 7 & \mathbf{13} & \mathbf{7} &
\mathbf{1} & 0 & \cdots \\1 & 9  & 25 & 25 & 9 & 1 &\cdots\\ \vdots
& \vdots &
\vdots & \vdots & \vdots & \vdots &
\ddots\end{array}\right),\end{displaymath} with central elements that begin
$$1,3,13,63,321,1683,\ldots.$$
The associated shifted central triangle begins
\begin{displaymath}c(A;1)=\left(\begin{array}{ccccccc} 1 & 0 & 0 & 0
&0 & 0 & \cdots \\3 & 1 & 0 & 0 & 0 & 0 & \cdots \\ 13 & 7 & 1 &
0 & 0 & 0 &
\cdots \\ 63 & 41 & 11 & 1 & 0 & 0 & \cdots \\ 321 & 231 & 85 & 15 &
1 & 0 & \cdots \\1683 & 1289  & 575 & 145 & 19 & 1 &\cdots\\ \vdots
& \vdots &
\vdots & \vdots & \vdots & \vdots &
\ddots\end{array}\right).\end{displaymath}
\end{example}

Then we have the following theorem.
\begin{theorem} The shifted central triangle $c(A;1)$ of the Riordan array $A=(g(x), xf(x))$ is a Riordan array which admits the following factorization.
$$ c(A;1)=\left(\frac{\phi'}{f(\phi)}, \phi\right)\cdot A,$$
where
$$\phi(x)=\textrm{Rev}\left(\frac{x}{f(x)}\right) \quad \textrm{and} \quad \phi'(x)=\frac{d}{dx}\phi(x).$$
\end{theorem}
\begin{proof}
The matrix $c(A;1)$ is a Riordan array since we exhibit it as the product of two Riordan arrays. In order to show that it is the product of two Riordan arrays, we proceed as follows, using Lagrange inversion.
\begin{eqnarray*}
a_{2n,n+k}&=& [x^{2n}] g(x)(xf(x))^{n+k}\\
&=&[x^{2n}] x^{n+k}g(x)f(x)^k f(x)^n\\
&=&[x^n] g(x)(xf(x))^k f(x)^n\\
&=&\sum_{i=0}^n [x^i] g(x)(xf(x))^k [x^{n-i}]f(x)^n\\
&=& \sum_{i=0}^n a_{i,k} [x^n] \frac{x^i}{f(x)} f(x)^{n+1}\\
&=& \sum_{i=0}^n a_{i,k} (n+1) \frac{1}{n+1} [x^n] F'(x) f(x)^{n+1} \quad (F'(x)=\frac{x^i}{f(x)})\\
&=& \sum_{i=0}^n a_{i,k} (n+1) [x^{n+1}] F\left(\textrm{Rev}\left(\frac{x}{f(x)}\right)\right)\\
&=& \sum_{i=0}^n a_{i,k} [x^n] F'\left(\textrm{Rev}\left(\frac{x}{f(x)}\right)\right)\frac{d}{dx} \textrm{Rev}\left(\frac{x}{f(x)}\right)\\
&=& \sum_{i=0}^n a_{i,k} [x^n] \frac{\left(\textrm{Rev}\left(\frac{x}{f(x)}\right)\right)^i}{f\left(\textrm{Rev}\left(\frac{x}{f(x)}\right)\right)} \frac{d}{dx} \textrm{Rev}\left(\frac{x}{f(x)}\right)\\
&=& \sum_{i=0}^n a_{i,k} [x^n] \frac{\phi'(x)}{f(\phi(x))} (\phi(x))^i\\
&=& \sum_{i=0}^n m_{n,i} a_{i,k}, \end{eqnarray*}
where
$$m_{n,k}=[x^n] \frac{\phi'(x)}{f(\phi(x))}(\phi(x))^k$$ is the general term of the Riordan array
$$\left(\frac{\phi'}{f(\phi)}, \phi\right).$$
\end{proof}
\begin{corollary} Let the generating function of the $A$-sequence of the Riordan array $A$ be $A(x)$. Then the generating function of the $A$-sequence of $c(A;1)$ is $A(x)^2$.
\end{corollary}
\begin{proof} Let the generating function of the $A$-sequence of the Riordan array $\left(\frac{\phi'}{f(\phi)}, \phi\right)$ be
$A_1(x)$ and let the generating function of the $A$-sequence of $A$ be $A(x)$. Then the generating function of the $A$-sequence of $c(A;1)$ is given by \cite{AZ}
$$ A(x) A_1\left(\frac{x}{A(x)}\right).$$
Now in our case we have
$$A_1(x)=\frac{x}{\textrm{Rev}(\phi)}=\frac{x}{\frac{x}{f(x)}}=f(x),$$ and
$$A(x)=\frac{x}{\textrm{Rev}(xf)}.$$
Thus the generating function of the $A$-sequence of $c(A;1)$ is given by
\begin{eqnarray*}
 A(x) A_1\left(\frac{x}{A(x)}\right)&=&\frac{x}{\textrm{Rev}(xf)}\cdot f\left(\frac{x}{\frac{x}{\textrm{Rev}(xf)}}\right)\\
 &=& \frac{x}{\textrm{Rev}(xf)} \cdot f(\textrm{Rev}(xf))\\
 &=& \frac{x}{\textrm{Rev}(xf)} \cdot \frac{\textrm{Rev}(xf) f(\textrm{Rev}(xf))}{\textrm{Rev}(xf)}\\
 &=& \frac{x}{\textrm{Rev}(xf)} \cdot \frac{(xf)(\textrm{Rev}(xf))}{\textrm{Rev}(xf)}\\
 &=& \frac{x}{\textrm{Rev}(xf)} \cdot\frac{x}{\textrm{Rev}(xf)}\\
 &=& A(x)^2.\end{eqnarray*}
 \end{proof}
\begin{corollary} We have
$$c(A;1)=\left(\phi' \frac{g(\phi)}{f(\phi)}, \phi f(\phi)\right).$$
\end{corollary}
\begin{proof}
From the above theorem, we have
\begin{eqnarray*}
c(A;1)&=&\left(\frac{\phi'}{f(\phi)}, \phi\right)\cdot (g(x), xf(x))\\
&=&\left(\frac{\phi'}{f(\phi)} g(\phi), \phi f(\phi)\right)\\
&=&\left(\phi' \frac{g(\phi)}{f(\phi)}, \phi f(\phi)\right)\end{eqnarray*}
\end{proof}
\begin{corollary} We have
$$c(A;1)=\left(\frac{f(x)}{\phi'\left(\frac{x}{f(x)}\right)}, \frac{x}{f(x)}\right)^{-1}\cdot (g(x),xf(x)).$$
\end{corollary}
\begin{proof}
We calculate the inverse of $\left(\frac{\phi'}{f(\phi)}, \phi\right)$.
We obtain
\begin{eqnarray*}
\left(\frac{\phi'}{f(\phi)}, \phi\right)^{-1}&=&\left(\frac{1}{\frac{\phi'(\bar{\phi})}{f(\phi(\bar{\phi}))}},\bar{\phi}\right)\\
&=&\left(\frac{f(x)}{\phi'(\bar{\phi})}, \frac{x}{f(x)}\right)\\
&=&\left(\frac{f(x)}{\phi'\left(\frac{x}{f(x)}\right)}, \frac{x}{f(x)}\right).\end{eqnarray*}
\end{proof}
\begin{corollary} We have
$$c(A;1)^{-1}=\left(\frac{1}{g(\bar{v})} \frac{f(\bar{v})}{\phi'\left(\frac{\bar{v}}{f(\bar{v})}\right)}, \frac{\bar{v}}{f(\bar{v})}\right),$$
where $v(x)=xf(x)$.
\end{corollary}
\begin{proof}
From the last result, we have
$$c(A;1)=\left(\frac{f(x)}{\phi'\left(\frac{x}{f(x)}\right)}, \frac{x}{f(x)}\right)^{-1}\cdot (g(x),v(x)).$$
Taking inverses, we get
$$c(A;1)^{-1}=(g(x), v(x))^{-1} \cdot \left(\frac{f(x)}{\phi'\left(\frac{x}{f(x)}\right)}, \frac{x}{f(x)}\right).$$ Thus we have
\begin{eqnarray*}
c(A;1)^{-1}&=&\left(\frac{1}{g(\bar{v})}, \bar{v}\right)\cdot \left(\frac{f(x)}{\phi'\left(\frac{x}{f(x)}\right)}, \frac{x}{f(x)}\right)\\
&=&\left(\frac{1}{g(\bar{v})} \frac{f(\bar{v})}{\phi'\left(\frac{\bar{v}}{f(\bar{v})}\right)}, \frac{\bar{v}}{f(\bar{v})}\right).\end{eqnarray*}
\end{proof}
Note that we have
$$ \frac{\bar{v}}{f(\bar{v})}=\frac{\bar{v}^2}{x}.$$
 A consequence of this result is that the $Z$-sequence of the matrix $c(A;1)$ has the following generating function.
$$Z(x;1)=\frac{A(x)^2}{x}\left(1-\frac{1}{\phi'\left(\frac{\bar{v}}{f(\bar{v})}\right)} \frac{f(\bar{v})}{g(\bar{v})}\right).$$
\begin{example}
We let $$A=\left(\frac{1}{1-x}, \frac{x(1+x)}{1-x}\right).$$ Then
$A(x)=\frac{1+x+\sqrt{1+6x+x^2}}{2}$. We obtain
$$c(A;1)^{-1}=\left(\frac{(1-x)(\sqrt{1+6x+x^2}-x-1)}{2x}, \frac{(\sqrt{1+6x+x^2}-x-1)^2}{4x}\right),$$
and
$$Z(x;1)=\frac{A(x)^2}{x}\left(1-\frac{(1-x)(\sqrt{1+6x+x^2}-x-1)}{2x}\right)=2+x+\sqrt{1+6x+x^2},$$ where
$Z(x;1)$ is the generating function for the $Z$-sequence of $c(A;1)$.
\end{example}
Our next result concerns elements of the Bell subgroup of Riordan arrays.
\begin{corollary} Let $A=(f(x), xf(x))$ be a member of the Bell subgroup of the Riordan group.
Then
$$c(A;1)=\left(\phi', \phi f(\phi)\right),$$ and
$$c(A;1)^{-1}=\left(\frac{1}{\phi'\left(\frac{\bar{v}}{f(\bar{v})}\right)}, \frac{\bar{v}}{f(\bar{v})}\right).$$
\end{corollary}
\begin{proof} In this case, $g(x)=f(x)$ we have
$$c(A;1)=\left(\phi' \frac{g(\phi)}{f(\phi)}, \phi f(\phi)\right)=\left(\phi', \phi f(\phi)\right),$$ and similarly for $c(A;1)^{-1}$.
\end{proof}
\begin{corollary} The generating function of the central elements $a_{2n,n}$ of $A$ is given by
$$\phi'(x)\frac{g(\phi(x))}{f(\phi(x))}.$$
\end{corollary}
This last result can also be found in \cite{Central}

\section{Examples}

\begin{example} We look again at the example of
$$A=\left(\frac{1}{1-x}, \frac{x(1+x)}{1-x}\right) \quad\quad (\seqnum{A008288}) \cite{Pascal, Del}.$$
For this, we have
$$g(x)=\frac{1}{1-x}, \quad f(x)=\frac{1+x}{1-x}, \quad \frac{x}{f(x)}=\frac{x(1-x)}{1+x},\quad, \phi(x)=\frac{1-x-\sqrt{1-6x+x^2}}{2}.$$
It follows that
$$\phi'(x)\frac{g(\phi(x))}{f(\phi(x))}=\frac{1}{\sqrt{1-6x+x^2}},$$
and
$$\phi(x)f(\phi(x))=\frac{(1-x-\sqrt{1-6x+x^2})^2}{4x}.$$
Hence we have
$$c(A;1)=\left(\frac{1}{\sqrt{1-6x+x^2}}, \frac{(1-x-\sqrt{1-6x+x^2})^2}{4x}\right).$$
Now let $B=\left(\frac{\phi'}{f(\phi)}, \phi\right)$, so that
$$c(A;1)= B \cdot A.$$
We find that
$$B^{-1}=\left(\frac{1-2x-x^2}{1-x^2}, \frac{x(1-x)}{1+x}\right),$$ so that
$$B=\left(\frac{1+x+\sqrt{1-6x+x^2}}{2\sqrt{1-6x+x^2}}, \frac{1-x-\sqrt{1-6x+x^2}}{2}\right).$$
Finally, we calculate $c(A;1)^{-1}$ as follows.
\begin{eqnarray*}
c(A;1)^{-1}&=& A^{-1}\cdot B^{-1}\\
&=& \left(\frac{3+x-\sqrt{1+6x+x^2}}{2}, \frac{\sqrt{1+6x+x^2}-x-1}{2}\right)\cdot \left(\frac{1-2x-x^2}{1-x^2}, \frac{x(1-x)}{1+x}\right)\\
&=& \left(\frac{(1-x)\sqrt{1+6x+x^2}+x^2-1}{2x}, \frac{1+4x+x^2-(1+x)\sqrt{1+6x+x^2}}{2x}\right).\end{eqnarray*}
\end{example}
\begin{example} We now take the example of
$$A=\left(\frac{1}{c(x)}, \frac{x}{c(x)}\right),$$
where $$c(x)=\frac{1-\sqrt{1-4x}}{2x}$$ is the generating function of the Catalan numbers
$C_n=\frac{1}{n+1}\binom{2n}{n}$.
For this example, we have
$$g(x)=\frac{1}{c(x)}, \quad f(x)=\frac{x}{c(x)}, \quad \frac{x}{f(x)}=xc(x),\quad \phi(x)=x(1-x).$$
Thus
$$\phi'(x)=\frac{d}{dx} x(1-x)=1-2x.$$
We find that
\begin{eqnarray*}c(A;1)&=&\left(\frac{1-2x}{1-x}, x(1-x)\right)\cdot \left(\frac{1}{c(x)}, \frac{x}{c(x)}\right)\\
&=& \left(\frac{1-2x}{1-x} \frac{1}{c(x(1-x))}, \frac{x(1-x)}{c(x(1-x))}\right)\\
&=& (1-2x, x(1-x)^2).\end{eqnarray*}
\end{example}
\begin{example} We now look at the Catalan matrix
$$A=(c(x),xc(x))\quad \quad (\seqnum{A033184}),$$ which begins
\begin{displaymath}A=\left(\begin{array}{ccccccc} 1 & 0 & 0 & 0
&0 & 0 & \cdots \\1 & 1 & 0 & 0 & 0 & 0 & \cdots \\ 2 & 2 & 1 &
0 & 0 & 0 &
\cdots \\ 5 & 5 & 3 & 1 & 0 & 0 & \cdots \\ 14 & 14 & 9 & 4 &
1 & 0 & \cdots \\42 & 42  & 28 & 14 & 5 & 1 &\cdots\\ \vdots
& \vdots &
\vdots & \vdots & \vdots & \vdots &
\ddots\end{array}\right),\end{displaymath} with central elements that begin
$$1, 2, 9, 48, 275, 1638, 9996, 62016,\ldots,$$ or $\frac{n+1}{2n+1}\binom{3n}{n}$ (\seqnum{A174687}). This triangle is \seqnum{A033184}, with general term $\frac{k+1}{n+1} \binom{2n-k}{n-k}$. We note that the
production matrix of $(c(x), xc(x))$ has production array that begins
\begin{displaymath}\left(\begin{array}{ccccccc} 1 & 1 & 0 & 0
&0 & 0 & \cdots \\1 & 1 & 1 & 0 & 0 & 0 & \cdots \\ 1 & 1 & 1 &
1 & 0 & 0 &
\cdots \\ 1 & 1 & 1 & 1 & 1 & 0 & \cdots \\ 1 & 1 & 1 & 1 &
1 & 1 & \cdots \\1 & 1  & 1 & 1 & 1 & 1 &\cdots\\ \vdots
& \vdots &
\vdots & \vdots & \vdots & \vdots &
\ddots\end{array}\right).\end{displaymath}  In particular, the $A$-sequence of $(c(x),xc(x))$ has
g.f. given by $\frac{1}{1-x}$.

The associated central triangle begins
\begin{displaymath}c(A;1)=\left(\begin{array}{ccccccc} 1 & 0 & 0 & 0
&0 & 0 & \cdots \\2 & 1 & 0 & 0 & 0 & 0 & \cdots \\ 9 & 4 & 1 &
0 & 0 & 0 &
\cdots \\ 48 & 20 & 6 & 1 & 0 & 0 & \cdots \\ 275 & 110 & 35 & 8 &
1 & 0 & \cdots \\1638 & 637  & 208 & 54 & 10 & 1 &\cdots\\ \vdots
& \vdots &
\vdots & \vdots & \vdots & \vdots &
\ddots\end{array}\right).\end{displaymath}
For this example, we have
$$g(x)=f(x)=c(x), \quad \frac{x}{f(x)}=\frac{x}{c(x)}, \quad \phi(x)=2\sqrt{\frac{x}{3}}\sin\left(\frac{1}{3}\sin^{-1}\left(\frac{\sqrt{27x}}{2}\right)\right).$$
We note that $\phi(x)$ is the generating function of $\frac{1}{2n-1} \binom{3n-3}{n-1}$ (\seqnum{A001764}).
We have
$$\phi'(x)=\frac{1}{\sqrt{4-27x}}\cos\left(\frac{1}{3}\sin^{-1}\left(\frac{\sqrt{27x}}{2}\right)\right)+
\frac{1}{\sqrt{3x}}\sin\left(\frac{1}{3}\sin^{-1}\left(\frac{\sqrt{27x}}{2}\right)\right).$$
Then $\phi'(x)$ is the generating function of the central coefficients of $A=(c(x),xc(x))$, or $\frac{n+1}{2n+1}\binom{3n}{n}$ (\seqnum{A006013}).
We then have
\begin{eqnarray*}\phi(x) f(\phi(x))&=&\frac{1}{2}-\frac{3^{\frac{3}{4}}}{6}\sqrt{\sqrt{3}-8\sqrt{x}\sin\left(\frac{1}{3}\sin^{-1}\left(\frac{\sqrt{27x}}{2}\right)\right)}\\
&=& \frac{2}{3}\left(1-\cos\left(\frac{1}{3} \cos^{-1}\left(\frac{2-27x}{2}\right)\right)\right)\\
&=& \frac{4}{3} \sin\left(\frac{1}{3} \sin^{-1} \left(\frac{\sqrt{27x}}{2}\right)\right)^2.\end{eqnarray*}
This is the generating function of $\frac{1}{n}\binom{3n-2}{n-1}$. It is the reversion of $x(1-x)^2$.
Thus we obtain \begin{scriptsize}
$$c(A;1)=\left(\frac{1}{\sqrt{4-27x}}\cos\left(\frac{1}{3}\sin^{-1}\left(\frac{\sqrt{27x}}{2}\right)\right)+
\frac{1}{\sqrt{3x}}\sin\left(\frac{1}{3}\sin^{-1}\left(\frac{\sqrt{27x}}{2}\right)\right),\frac{1}{2}-\frac{3^{\frac{3}{4}}}{6}\sqrt{\sqrt{3}-8\sqrt{x}\sin\left(\frac{1}{3}\sin^{-1}\left(\frac{\sqrt{27x}}{2}\right)\right)}\right).$$
\end{scriptsize}
The $A$ sequence of $c(A;1)$ has generating function $\frac{1}{(1-x)^2}$. The production matrix of $c(A;1)$ begins
\begin{displaymath}\left(\begin{array}{ccccccc} 2 & 1 & 0 & 0
&0 & 0 & \cdots \\5 & 2 & 1 & 0 & 0 & 0 & \cdots \\ 10 & 3 & 2 &
1 & 0 & 0 &
\cdots \\ 19 & 4 & 3 & 2 & 1 & 0 & \cdots \\ 36 & 5 & 4 & 3 &
2 & 1 & \cdots \\69 & 6  & 5 & 4 & 3 & 2 &\cdots\\ \vdots
& \vdots &
\vdots & \vdots & \vdots & \vdots &
\ddots\end{array}\right).\end{displaymath}
\end{example}

\section{The second shifted central triangle}
In this section, we study the triangle $c(A;2)$ with general term $a_{2n+1,n+k+1}$. We have the following theorem.
\begin{theorem} The $2$-shifted central triangle $c(A;2)$ of the Riordan array $A=(g(x),xf(x))$ is a Riordan array which admits the factorization

$$c(A;2)=(\phi', \phi)\cdot A,$$ where
$$\phi(x)=\textrm{Rev}\left(\frac{x}{f(x)}\right), \quad \textrm{and}\quad \phi'(x)=\frac{d}{dx}\phi(x).$$
\end{theorem}
\begin{proof}
We have
\begin{eqnarray*}
a_{2n+1,n+k+1}&=&[x^{2n+1}] g(x)(xf(x))^{n+k+1}\\
&=& [x^{2n+1}] x^{n+k+1} g(x) f(x)^k f(x)^{n+1}\\
&=& [x^n] g(x) (xf(x))^k f(x)^{n+1}\\
&=& \sum_{i=0}^n [x^i] g(x)(xf(x))^k [x^{n-i)} f(x)^{n+1}\\
&=& \sum_{i=0}^n a_{i,k} [x^n] x^i f(x)^{n+1}\\
&=& \sum_{i=0}^n a_{i,k} (n+1)\frac{1}{n+1}[x^n]F'(x)f(x)^{n+1} \quad (F'(x)=x^i)\\
&=& \sum_{i=0}^n a_{i,k} (n+1) [x^{n+1}] F\left(\textrm{Rev}\left(\frac{x}{f(x)}\right)\right)\\
&=& \sum_{i=0}^n a_{i,k} [x^n] F'\left(\textrm{Rev}\left(\frac{x}{f(x)}\right)\right) \frac{d}{dx}\textrm{Rev}\left(\frac{x}{f(x)}\right)\\
&=& \sum_{i=0}^n a_{i,k} [x^n] \left(\textrm{Rev}\left(\frac{x}{f(x)}\right)\right)^i \frac{d}{dx}\textrm{Rev}\left(\frac{x}{f(x)}\right)\\
&=& \sum_{i=0}^n a_{i,k} [x^n] \phi'(x) (\phi(x))^i \\
&=& \sum_{i=0}^n \mu_{n,i} a_{i,k},\end{eqnarray*}
where
$$\mu_{n,k}=[x^n] \phi'(x) (\phi(x))^k$$ is the general term of the Riordan array
$$(\phi', \phi).$$
\end{proof}
\begin{corollary}
We have
$$c(A;2)=(\phi' g(\phi), \phi f(\phi)).$$
\end{corollary}
\begin{corollary}
The Riordan arrays $c(A;1)$ and $c(A;2)$ are related by
$$ c(A;2)=(f(\phi),x)\cdot c(A;1).$$
\end{corollary}
Thus we have
$$c(A;2)\cdot c(A;1)^{-1}=(f(\phi), x).$$
We now look at the product
$$c(A;1)^{-1}\cdot c(A;2).$$
\begin{proposition}
We have
$$c(A;1)^{-1}\cdot c(A;2)=\left(\frac{x}{\phi\left(\overline{\phi(f(\phi))}\right)}, x\right).$$
\end{proposition}
\begin{proof}
We have
\begin{eqnarray*}
c(A;1)^{-1}\cdot c(A;2)&=&\left(\phi' \frac{g(\phi)}{f(\phi)}, \phi(f(\phi)\right)^{-1}\cdot (\phi' g(\phi), \phi(f(\phi)))\\
&=&\left(\frac{1}{\phi'(\overline{\phi f(\phi)}) \frac{g(\phi(\overline{\phi f(\phi)}))}{f(\phi(\overline{\phi f(\phi)}))}}, \overline{\phi(f(\phi))} \right) \cdot (\phi' g(\phi), \phi(f(\phi)))\\
&=& \left(\frac{f(\phi(\overline{\phi(f(\phi))}))}{\phi'(\overline{\phi(f(\phi))}) g(\phi(\overline{\phi(f(\phi))}))}\cdot \phi'(\overline{\phi(f(\phi))}) g(\phi(\overline{\phi(f(\phi))})), x\right)\\
&=& (f(\phi(\overline{\phi(f(\phi))})), x)\\
&=& \left(\frac{1}{\phi(\overline{\phi(f(\phi))})} \cdot \phi(\overline{\phi(f(\phi))})f(\phi(\overline{\phi(f(\phi))})),x\right)\\
&=& \left(\frac{x}{\phi\left(\overline{\phi(f(\phi))}\right)}, x\right).
\end{eqnarray*}
Thus
$$c(A;1)^{-1} \cdot c(A;2)=\left(\frac{x}{\phi\left(\overline{\phi(f(\phi))}\right)}, x\right).$$
\end{proof}
\begin{example}
We take the case of the binomial triangle
$$A=\left(\frac{1}{1-x}, \frac{x}{1-x}\right),$$ with general element
$$a_{n,k}=\binom{n}{k}.$$
Thus
$$c(A;1)=\left(\binom{2n}{n+k}\right),\quad \textrm{and} \quad c(A;2)=\left(\binom{2n+1}{n+k+1}\right).$$
We have $f(x)=\frac{1}{1-x}$ and so $\frac{x}{f(x)}=x(1-x)$. Thus
$$\phi(x)=xc(x),$$ and
$$\phi(x)f(\phi(x))=\frac{1-2x-\sqrt{1-4x}}{2x} \Rightarrow \textrm{Rev}(\phi f(\phi))=\frac{x}{(1+x)^2}.$$
Then
$$\frac{x}{\phi\left(\frac{x}{(1+x)^2}\right)}=1+x.$$
Thus in this case
$$c(A;1)^{-1} \cdot c(A;2)= (1+x,x).$$
Now
$$f(\phi(x))=\frac{1}{1-xc(x)}=c(x),$$ so that
$$c(A;2) \cdot c(A;1)^{-1}= (c(x),x).$$
\end{example}
We end this section with the following observation about a special conjugation.
\begin{proposition}
We have
$$c(A;1)^{-1}\cdot c(A;2) \cdot c(A;1)^{-1}=\left(\phi' \frac{g(\phi)}{f(\phi)^2}, \phi f(\phi)\right)^{-1}.$$
\end{proposition}
\begin{proof} We have
\begin{eqnarray*}
c(A;1)^{-1}\cdot c(A;2) \cdot c(A;1)^{-1}&=& \left(\frac{x}{\phi(s)},x\right)\cdot c(A;1)^{-1} \quad \textrm{where }s(x)=\overline{\phi(x) f(\phi(x))}\\
&=& \left(\frac{x}{\phi(s)},x\right) \cdot \left(\phi' \frac{g(\phi)}{f(\phi)}, \phi f(\phi)\right)^{-1}.\end{eqnarray*}
Hence
\begin{eqnarray*}
\left(c(A;1)^{-1}\cdot c(A;2) \cdot c(A;1)^{-1}\right)^{-1}&=&\left(\phi' \frac{g(\phi)}{f(\phi)}, \phi f(\phi)\right) \cdot \left(\frac{\phi(s)}{x},x\right)\\
&=&\left(\phi' \frac{g(\phi)}{f(\phi)} \frac{\phi(x)}{\phi(x)f(\phi(x))}, \phi f(\phi)\right)  \quad \textrm{since }s(\phi(x)f(\phi(x)))=x\\
&=& \left(\phi' \frac{g(\phi)}{f(\phi)^2}, \phi f(\phi)\right).\end{eqnarray*}
Alternatively we note that
$$c(A;1)^{-1}\cdot c(A;2) \cdot c(A;1)^{-1}=c(A;1)^{-1} \cdot (f(\phi), x),$$ and hence
$$\left(c(A;1)^{-1}\cdot c(A;2) \cdot c(A;1)^{-1}\right)^{-1}=\left(\frac{1}{f(\phi)},x\right)\cdot c(A;1)=\left(\phi' \frac{g(\phi)}{f(\phi)^2}, \phi f(\phi)\right).$$
\end{proof}
\begin{example}
We take the example of $A=\left(\frac{1}{1-x}, \frac{x(1-x)}{1+x}\right)$.
Then we find that
$$c(A;1)^{-1}\cdot c(A;2) \cdot c(A;1)^{-1}=\left(\frac{1-x+\sqrt{1-6x+x^2}}{2\sqrt{1-6x+x^2}}, \frac{1-4x+x^2-(1-x)\sqrt{1-6x+x^2}}{2x}\right),$$ which begins
\begin{displaymath}\left(\begin{array}{ccccccc} 1 & 0 & 0 & 0
&0 & 0 & \cdots \\1 & 1 & 0 & 0 & 0 & 0 & \cdots \\ 5 & 5 & 1 &
0 & 0 & 0 &
\cdots \\ 25 & 25 & 9 & 1 & 0 & 0 & \cdots \\ 129 & 129 & 61 & 13 &
1 & 0 & \cdots \\681 & 681  & 377 & 113 & 17 & 1 &\cdots\\ \vdots
& \vdots &
\vdots & \vdots & \vdots & \vdots &
\ddots\end{array}\right).\end{displaymath}
The first column is \seqnum{A002002}, the number of peaks in all Schroeder paths from $(0,0)$ to $(2n,0)$. The matrix obtained by removing this first column is just $c(A;2)$.
\end{example}

\section{The General Case}
We now look at the general case of the $r$-shifted central triangle
$$c(A;r)=\left(a_{2n+r,n+k+r}\right).$$
\begin{theorem}
Let $A=(g(x), xf(x))$ be a Riordan array and let $c(A;r)$ be the matrix with general element $a_{2n+r,n+k+r}$. Then $c(A;r)$ is a Riordan array that has the following factorization.
$$c(A;r)= \left(\phi'(x) f(\phi(x))^{r-1}, \phi(x)\right)\cdot A.$$
\end{theorem}
\begin{proof}
We have
\begin{eqnarray*}
a_{2n+r,n+k+r}&=& [x^{2n+r}] g(x)(xf(x))^{n+k+r}\\
&=&[x^{2n}] x^{n+k+r}g(x)f(x)^k f(x)^{n+r}\\
&=&[x^n] g(x)(xf(x))^k f(x)^{n+r}\\
&=&\sum_{i=0}^n [x^i] g(x)(xf(x))^k [x^{n-i}]f(x)^{n+r}\\
&=& \sum_{i=0}^n a_{i,k} [x^n] x^i f(x)^{r-1} f(x)^{n+1}\\
&=& \sum_{i=0}^n a_{i,k} (n+1) \frac{1}{n+1} [x^n] F'(x) f(x)^{n+1} \quad (F'(x)=x^i f(x)^{r-1})\\
&=& \sum_{i=0}^n a_{i,k} (n+1) [x^{n+1}] F\left(\textrm{Rev}\left(\frac{x}{f(x)}\right)\right)\\
&=& \sum_{i=0}^n a_{i,k} [x^n] F'\left(\textrm{Rev}\left(\frac{x}{f(x)}\right)\right)\frac{d}{dx} \textrm{Rev}\left(\frac{x}{f(x)}\right)\\
&=& \sum_{i=0}^n a_{i,k} [x^n] \left(\textrm{Rev}\left(\frac{x}{f(x)}\right)\right)^i f\left(\textrm{Rev}\left(\frac{x}{f(x)}\right)\right)^{r-1} \frac{d}{dx} \textrm{Rev}\left(\frac{x}{f(x)}\right)\\
&=& \sum_{i=0}^n a_{i,k} [x^n] \phi'(x)f(\phi(x))^{r-1} (\phi(x))^i\\
&=& \sum_{i=0}^n m_{n,i} a_{i,k}, \end{eqnarray*}
where
$$m_{n,k}=[x^n] \phi'(x)f(\phi(x))^{r-1}(\phi(x))^k$$ is the general term of the Riordan array
$$\left(\phi'f(\phi)^{r-1}, \phi\right).$$
\end{proof}
\begin{corollary}
The generating function of the $A$-sequence of $c(A;r)$ is $A(x)^2$.
\end{corollary}
\begin{proof} This follows as in the case of $c(A;1)$ due to the above factorization.
\end{proof}
\begin{corollary}
$$c(A;r)=(\phi' g(\phi) f(\phi)^{r-1}, \phi f(\phi)).$$
\end{corollary}
\begin{corollary} We have
$$c(A;r)=\left(\frac{1}{\phi'(\frac{x}{f(x)})f(x)^{r-1}}, \frac{x}{f(x)}\right)^{-1}\cdot A.$$
\end{corollary}
\begin{proof}
We have
\begin{eqnarray*}c(A;r)&=&\left(\frac{1}{\phi'(\bar{\phi})f(x)^{r-1}}, \frac{x}{f(x)}\right)^{-1}\cdot A\\
&=&\left(\frac{1}{\phi'(\frac{x}{f(x)})f(x)^{r-1}}, \frac{x}{f(x)}\right)^{-1}\cdot A.\end{eqnarray*}
\end{proof}
As in the case of $c(A;1)$ we can then show the following.
\begin{corollary} We have
$$c(A;r)^{-1}=\left( \frac{1}{\phi'\left(\frac{x}{f(x)}\right)g(\bar{v}) f(\bar{v})^{r-1}}, \frac{\bar{v}}{f(\bar{v})}\right).$$
\end{corollary}
We deduce from this last result that the generating function $Z(A;r)$ of the $Z$-sequence of $c(A;r)$ is given by
$$Z(A;r)=\frac{A(x)^2}{x}\left(1-\frac{1}{\phi'\left(\frac{x}{f(x)}\right)g(\bar{v}) f(\bar{v})^{r-1}}\right).$$
As in the last section, we can also show the following.
\begin{proposition}
$$c(A;r)^{-1} \cdot c(A;r+1)=\left(\frac{x}{\phi\left(\overline{\phi(f(\phi))}\right)}, x\right),$$ and
$$c(A;r+1)\cdot c(A;r)^{-1}=(f(\phi), x).$$
\end{proposition}

\bigskip
\hrule
\bigskip
\noindent 2010 {\it Mathematics Subject Classification}: Primary
15B36; Secondary 11B83, 11C20.
\noindent \emph{Keywords:} Riordan group, Riordan array, central coefficients, Lagrange inversion.

\end{document}